\documentclass[11pt, reqno]{amsart}

\usepackage{amssymb,latexsym}

\usepackage{enumerate}

\newtheorem{thm}{ \bf Theorem}[section]

\newtheorem{lem}[thm]{ \bf Lemma}

\numberwithin{equation}{section}

\begin{document}

\baselineskip=17pt

\title[ Zero-Sum Games ]
{ Zero-Sum Stochastic Games with Partial Information and Average Payoff}
\thanks{This work is supported in part by SPM fellowship of
CSIR and in part by UGC Centre for Advanced Study.}

\author[Subhamay Saha]{ Subhamay Saha}
\address{Department of Mathematics\\
Indian Institute of Science\\
Bangalore 560 012, India.}
\email{subhamay@math.iisc.ernet.in}

%---------------------------------------------------------------------------

\date{}

\begin{abstract}
We consider discrete time partially observable zero-sum stochastic
game with average payoff criterion. We study the game using an
equivalent completely observable game. We show that the game has a
value and also we come up with a pair of optimal strategies for
both the players.
\end{abstract}

%---------------------------------------------------------------------------

\subjclass[2000]{Primary 91A15 ; Secondary 91A05, 91A25.}

\keywords{Stochastic games, partial observation, average payoff, saddle point strategies.}

%---------------------------------------------------------------------------

\maketitle

\section{\textbf{Introduction}}
Stochastic games were introduced by Shapley in \cite{shapely}. Following this pioneering work there has been a lot of work on stochastic games.
For a survey on zero-sum games we refer to \cite{vrieze}. Most of the available literature in this category concerns stochastic games
with complete observation, i.e., at each stage, the state of the game is completely known to the players. Although there is
considerable amount of literature (see \cite{shreve}, \cite{borkar}, \cite{yushkevich} and the references therein) available on
 partially observable Markov decision processes (\textbf{POMDP}) of which stochastic games are a generalisation, the corresponding
 literature in partially observable stochastic games is rather sparse. In \cite{ghosh} the authors
 study zero-sum games for partially observable stochastic games under discounted payoff criterion.
 In this article we investigate the same problem with the average payoff criterion. In \cite{borkar} the authors
 study \textbf{POMDP} under the average cost criteria using the approach based on Athreya-Ney-Nummelin
 construction of pseudo-atoms (\cite{athreya}, \cite{nummelin}) as described in \cite{meyn}. In this article
 we extend those ideas to the zero-sum game case. Zero-sum stochastic games are generally studied by solving
 the corresponding dynamic programming or Shapely equations \cite{vrieze}. This approach has also been carried out for partially
  observable games in \cite{ghosh}. In this paper instead of solving the appropriate Shapely equations we solve two dynamic programming type inequalities,
   which in turn lead to the existence of a value and saddle point strategies. Also our article extends the idea of using the pseudo-atom approach in solving MDP, to the stochastic game setup.
Under certain Lyapunov assumption we use the pseudo-atom construction to carry out a coupling argument,
which gives us appropriate bound on the relative $\alpha-$discounted value function. This bound then enables us to make appropriate limiting arguments.

The rest of the paper is organized as follows. In Section 2 we describe the model. In Section 3 we use the vanishing discount approach to prove the existence of a value and a saddle-point equilibrium for the \textbf{POSG}. We conclude with a few remarks in Section 4.

\section{\textbf{Preliminaries and Model Description}}
Let $X, Y$ and $U, V$ be Polish spaces representing state, observation and action spaces for player 1 and player 2 respectively. We further assume that $U$ and $V$ are compact. For any Polish space $S$, we denote by $\mathcal{P}(S)$ the Polish space of probability measures on $S$ and by $\mathcal{B}(S)$ the Borel $\sigma$-field on $S$. Let $\{X_n\}$ be an $X$-valued partially observed controlled Markov chain with $Y$-valued observation process $\{Y_n\}$. Let $$(x,u,v) \in X\times U \times V \rightarrow p(dz,dy| x,u,v) \in \mathcal{P}(X \times Y)$$ be a transition kernel which is assumed to be continuous in its arguments. Let $\lambda$ denote a regular Borel radon measure on $X$. We assume the existence of a probability measure $\eta$ on $Y$ and a $\varphi \in C_b(X \times U \times V \times X \times Y)$, with $\varphi(\cdot)>0$ such that $$ p(dz,dy| x,u,v)= \varphi(x,u,v,z,y)\lambda(dz) \eta(dy)\,.$$
The chain is controlled by two players. The first player chooses his actions from $U$ and player 2 chooses his actions from $V$. Let $\{U_n\}$ be an $U$-valued control sequence of player 1 and $\{V_n\}$ be a $V$-valued control sequence of player 2. The transition probability function of the controlled chain $\{X_n\}$ together with the observation chain $\{Y_n\}$ is given by
\begin{align*}\mathbb{P}(X_{n+1}\in A, Y_{n+1}\in B|X_m,Y_m,U_m,V_m,m\leq n)=\int_A\int_B\varphi(X_n,U_n,V_n,z,y)\lambda(dz),\eta(dy)\end{align*}
for $A \in \mathcal{B}(X)$ and $B \in \mathcal{B}(Y)$. The partially observed stochastic game (\textbf{POSG}) under ergodic payoff criteria is the following:

\noindent (i) The initial distribution of the (unobservable) state process is $\psi$ which is known to both the players; $Y_0$ is deterministic, say $Y_0=y^*$ for some fixed element $y^*$ in $Y$.

\noindent (ii) At the $0$th epoch the players based on the knowledge that the initial distribution of the state process is $\psi$, independently choose actions $u_0\in U$ and $v_0 \in V$. Consequently, conditional on the event $X_0=x_0$ player 1 gets an (unobservable) payoff $c(x_0, u_0, v_0)$ from player 2. Here $$c:X \times U \times V \rightarrow \mathbb{R}_+$$ is assumed to be a bounded continuous function. The next state and observation pair $(X_1,Y_1)$ is generated according to the stochastic kernel  $p(dz,dy| x_0,u_0,v_0)$.

\noindent (iii) Now conditioned on the event $Y_1=y_1$ the players again choose their actions and so on. This process is repeated over an infinite time horizon.

\noindent (iv) Each player can recall at any time the observations and actions of the past.

We now construct a probability space on which all the random variables are defined. the canonical sample space is defined as$$\Omega := (X \times Y \times U \times V)^{\infty}\,.$$ A generic element is of the form $$\omega = (x_0,y_0, u_0, v_0, x_1,\cdots)\,\,,\,\, x_i\in X, y_i \in Y, u_i \in U, v_i \in V\,.$$ The history spaces are defined as $$H_0=X \times Y, \,\, H_{n+1}:= H_n \times U \times V \times X \times Y\,.$$ The state, observation, actions and history processes denoted by $\{X_n\}, \{Y_n\}, \{U_n\}, \{V_n\}, \{H_n\}$ respectively are defined by the projections
\begin{align*} &X_n(\omega)=x_n \,\,\, Y_n(\omega)=y_n \\
&U_n(\omega)=u_n \,\,\, V_n(\omega)=v_n \\
&H_n(\omega)=(x_0,y_0, u_0, v_0,\cdots, u_{n-1}, v_{n-1}, x_n, y_n)\,.
\end{align*}
The entire history up to time $n$ is not available to the players for decision making at time $n$. The players have to make their decisions based on the observed history or information vector $$i_n:=(y_0, u_0, v_0,\cdots, u_{n-1}, v_{n-1}, y_n)$$ and the initial distribution $\psi$. We define the information spaces as follows:
$$I_0:= Y, \,\, I_{n+1}:= I_n \times U \times V \times Y\,.$$ The information process is defined by $$I_n(\omega)=(y_0,u_0,v_0,\cdots,u_{n-1},v_{n-1},y_n)\,.$$ An admissible strategy for player 1 is a sequence $\pi^1=\{\pi_n^1\}$ of stochastic kernels on $U$ given $\mathcal{P}(X)\times I_n$. The set of admissible strategies for player 1 is denoted by $\Pi^1$. Similarly an admissible strategy for player 2 is a sequence $\pi^2=\{\pi_n^2\}$ of stochastic kernels on $V$ given $\mathcal{P}(X)\times I_n$. The set of admissible strategies for player 2 is denoted by $\Pi^2$. With $\psi$ in $\mathcal{P}(X)$ and a pair of admissible strategies $(\pi^1,\pi^2) \in \Pi^1 \times \Pi^2$ specified, there exists a unique probability measure $\mathbb{P}_{\psi}^{\pi^1,\pi^2}$ on $(\Omega, \mathcal{B}(\Omega))$ defined by
\begin{align}\nonumber&\mathbb{P}_{\psi}^{\pi^1,\pi^2}(dx_0,dy_0,du_0,dv_0,\cdots,du_{n-1},dv_{n-1}, dx_n,dy_n)\\
&=\psi(dx_0)\delta_{y^*}(dy_0)\pi_0^1(du_0|\psi,y_0)\pi_0^2(dv_0|\psi,y_0)p(dx_1,dy_1|x_0,u_0,v_0)\cdots\\\nonumber&\pi^1_{n-1}(du_{n-1}|\psi, y_0,u_0,v_0,\cdots,y_{n-1})\pi^2_{n-1}(du_{n-1}|\psi, y_0,u_0,v_0,\cdots,y_{n-1})p(dx_n,dy_n|x_{n-1},u_{n-1},v_{n-1})\,.
\end{align}
We now describe the payoff criterion. Given the initial distribution $\psi$ and a pair of strategies $(\pi^1,\pi^2) \in \Pi^1 \times \Pi^2$, the average payoff criterion is given by
\begin{eqnarray} V_{\pi^1,\pi^2}(\psi)= \liminf_{n \rightarrow \infty}\frac{1}{n}\mathbb{E}_{\psi}^{\pi^1,\pi^2}\sum_{k=0}^{n-1}c(X_k, U_k, V_k)
\end{eqnarray} where $\mathbb{E}_{\psi}^{\pi^1,\pi^2}$ is the expectation with respect to the probability measure $\mathbb{P}_{\psi}^{\pi^1,\pi^2}$. Player 1 wishes to maximise $V_{\pi^1,\pi^2}(\psi)$ over all his admissible strategies and player 2 wishes to minimise the same over all his admissible strategies. A strategy ${\pi^*}^1$ is said to be optimal for player 1 if $$V_{{\pi^*}^1,\pi^2}(\psi) \geq \inf_{\Pi^2}\sup_{\Pi^1}V_{\pi^1,\pi^2}(\psi)$$ for any $\pi^2 \in \Pi^2$. Similarly a strategy ${\pi^*}^2$ is said to be optimal for player 2 if
$$V_{\pi^1,{\pi^*}^2}(\psi) \leq \sup_{\Pi^1}\inf_{\Pi^2}V_{\pi^1,\pi^2}(\psi)$$ for any $\pi^1 \in \Pi^1$. The game is said to have a value if
$$\inf_{\Pi^2}\sup_{\Pi^1}V_{\pi^1,\pi^2}(\psi)= \sup_{\Pi^1}\inf_{\Pi^2}V_{\pi^1,\pi^2}(\psi)\,.$$ If a pair of optimal strategies $({\pi^*}^1, {\pi^*}^2)$ exists for both the players then the pair $({\pi^*}^1, {\pi^*}^2)$ is called a saddle point equilibrium.
Now since the original state process is unobservable we define another state variable which is observable to the players. In order to achieve that, we have by conditioning
\begin{eqnarray}\label{payoff}V_{\pi^1,\pi^2}(\psi)= \liminf_{n\rightarrow \infty}\frac{1}{n}\sum_{m=0}^{n-1}\mathbb{E}_{\psi}^{\pi^1,\pi^2}[\tilde{c}(\Psi_m,U_m,V_m)]\,,\end{eqnarray}where $\{\Psi_n\}$ is the regular conditional law of $X_n$ given $I_n$, satisfying the recursion
\begin{align}\label{filter}\Psi_{n+1}(dz)=\frac{\int_X \Psi_n(dx)\varphi(x,U_n,V_n,z,Y_{n+1})\lambda(dz)}{\int_X\int_X \Psi_n(dx)\varphi(x,U_n,V_n,z,Y_{n+1})\lambda(dz)}\,, \,\,\,n\geq 0
\end{align}
and $$\tilde{c}(\psi,u,v)=\int_Xc(x,u,v)\psi(dx)\,.$$ Equation \eqref{filter} is known as the filtering equation. Note that since $Y_0$ is deterministic, $\Psi_0=$ the law of $X_0$. This allows us to consider an equivalent stochastic game with $\mathcal{P}(X)$-valued state process $\{\Psi_n\}$ with its evolution given by \eqref{filter}, under the same set of admissible strategies and with the payoff criterion given by \eqref{payoff}. This is a completely observable stochastic game (\textbf{COSG}) because $\Psi_n$ is known to both the players via the information upto time $n$. Thus we can solve the original \textbf{POSG} by solving this equivalent \textbf{COSG}.
Now in order to show that the \textbf{POSG} model under the average payoff criterion has a saddle point equilibrium and a value we impose the following Lyapunov type assumptions on our model.

\noindent\textbf{(A1)} There exists inf-compact functions $h$ and $\mathcal{V} \in C(X)$ satisfying $h\geq 1$, such that under any pair of admissible strategies and for any initial distribution \begin{eqnarray}\label{assump}\mathbb{E}(\mathcal{V}(X_{n+1})|\mathcal{F}_n)-\mathcal{V}(X_n)\leq -h(X_n)+cI_K
(X_n)
\end{eqnarray} where $K$ is some compact set with $\lambda(K)>0$ and $\mathcal{F}_n=\sigma(X_k,Y_k,U_k,V_k,k\leq n)$. We have dropped the super- and subscripts on $\mathbb{E}$ for notational convenience. Let $$\tau_{K}=\min\{n\geq 0: X_n \in K\}\,.$$ Then it is well known that (\cite{meyn}) $$\mathbb{E}[\tau_{K}|X_0=x]=O(\mathcal{V}(x))\,.$$
Define $$\mathcal{P}_0(X)=\{\mu \in \mathcal{P}(X): \int \mathcal{V}d\mu < \infty\}\,.$$ Now using \eqref{assump} we obtain
\begin{align*}&\mathbb{E}[\mathcal{V}(X_{n+1})]=\mathbb{E}[\int_X\mathcal{V}(x)d\Psi_{n+1}(dx)]\\&\leq \mathbb{E}[\mathcal{V}(X_n)]+\, \mbox{constant}\\
&=\mathbb{E}[\int_X\mathcal{V}(x)d\Psi_n(dx)]+\, \mbox{constant}\,.
\end{align*}
Hence it follows that if $\Psi_0 \in \mathcal{P}_0(X)$ then $\Psi_n \in \mathcal{P}_0(X),\,\forall n \geq 1$. We assume that $\Psi_0\in \mathcal{P}_0(X)$ and hence $\{\Psi_n\}$ can be viewed as a $\mathcal{P}_0(X)$-valued process. We further assume that

\noindent \textbf{(A2)} Under all admissible strategies  and for any initial distribution $$\lim_{n \rightarrow \infty}\frac{\mathbb{E}[\mathcal{V}(X_n)]}{n}=0\,.$$

\section{Saddle Point Strategies and Value}
We follow the vanishing discount approach to solve the average cost problem. Let $\alpha \in (0,1)$. Then consider the following discounted payoff \textbf{POSG}:
$$V_{\alpha}^{\pi^1,\pi^2}(\psi)=\mathbb{E}_{\psi}^{\pi^1, \pi^2}[\sum_{k=0}^{\infty}\alpha^k c(X_k,U_k,V_k)]$$ Player 1 tries to maximise the above quantity over all his admissible strategies and player 2 tries to minimise the same quantity over his admissible strategies. The definitions for the value of the game and for the optimal strategies can be given analogous to that of average payoff criterion. The following theorem can be proved using the equivalence with the \textbf{COSG} as discussed above and  standard arguments as in \cite{ghosh}:
\begin{thm} \label{discounted}The discounted payoff \textbf{POSG} has a value and the value function $V_{\alpha}(.)$ is the unique bounded solution of the following pair of Shapley equations:
\begin{align}\label{spe}\nonumber V_{\alpha}(\psi)=&\min_{\nu \in \mathcal{P}(V)}\max_{\mu \in \mathcal{P}(U)}\biggl[\bar{\tilde{c}}(\psi,\mu,\nu)
+ \alpha \int_{\mathcal{P}_0(X)}V_{\alpha}(\psi^{\prime})\phi(d\psi^{\prime}|\psi,\mu,\nu)\biggr]\\
=&\max_{\mu \in \mathcal{P}(U)}\min_{\nu \in \mathcal{P}(V)}\biggl[\bar{\tilde{c}}(\psi,\mu,\nu)
+ \alpha \int_{\mathcal{P}_0(X)}V_{\alpha}(\psi^{\prime})\phi(d\psi^{\prime}|\pi,\mu,\nu)\biggr]
\end{align}where $$\phi(d\psi^{\prime}|\psi,\mu,\nu)= \int_{U}\int_V \tilde{\phi}(d\psi^{\prime}|\psi,u,v)\mu(du)\nu(dv)$$ with  $\tilde{\phi}(d\psi^{\prime}|\psi,u,v)$ being the controlled transition kernel of the Markov chain $\{\Psi_n\}$, and
$$\bar{\tilde{c}}(\psi,\mu,\nu)= \int_U\int_V \tilde{c}(\psi,u,v)\mu(du)\nu(dv)\,.$$ Moreover let $u^*:\mathcal{P}_0(X)\rightarrow \mathcal{P}(U)$ be a measurable function such that $u^*(.)$ is an outer maximiser of \eqref{spe} then the strategy $\{{\pi^*_n}^1\}$ given by $${\pi^*_n}^1(\cdot|i_n)=u^*(\psi_n)(\cdot)\,.$$ is optimal for player 1. Further let $v^*:\mathcal{P}_0(X)\rightarrow \mathcal{P}(V)$ be a measurable function such that $v^*(.)$ is an outer minimiser of \eqref{spe} then $\{{\pi^*_n}^2\}$ given by $${\pi^*_n}^2(\cdot|i_n)=v^*(\psi_n)(\cdot)\,.$$ is an optimal strategy for player 2.
\end{thm}
 Now for the vanishing discount approach we need to compare $V_{\alpha}(.)$ for two different values of its argument. For that we construct on a common probability space two $X$-valued controlled Markov chains as above, controlled by the same pair of strategies but with different initial distributions $\hat{\psi}$ and $\tilde{\psi}$. This is done by a modification of the construction in the previous section. Let $\{\pi^1_n\}$ be an admissible strategy for player 1 and let $\{\pi^2_n\}$ be an admissible strategy for player 2. Define
 $$\bar{\Omega}=(X \times X \times Y \times Y \times U \times V)^{\infty}$$ with $\bar{\mathcal{F}}$ being the corresponding product Borel $\sigma$-algebra. Define $\bar{\mathbb{P}}^{\pi^1,\pi^2}_{\hat{\psi},\tilde{\psi}}$, a probability measure on $(\bar{\Omega}, \bar{\mathcal{F}})$ by
{\small\begin{align*}&\bar{\mathbb{P}}^{\pi^1,\pi^2}_{\hat{\psi},\tilde{\psi}}(d\hat{x}_0,d\tilde{x}_0,d\hat{y}_0,d\tilde{y}_0,du_0,dv_0,
d\hat{x}_1,d\tilde{x}_1,d\hat{y}_1,d\tilde{y}_1,du_1,dv_1,\cdots,du_{n-1},dv_{n-1},d\hat{x}_n,d\tilde{x}_n,d\hat{y}_n,d\tilde{y}_n)
\\&=\hat{\psi}(d\hat{x}_0)\tilde{\psi}(d\tilde{x}_0)\delta_{y^*}(d\hat{y}_0)\delta_{y^*}(d\tilde{y}_0)\pi^1_0(du_0|\hat{\psi},\hat{y}_0)
\pi^2_0(dv_0|\tilde{\psi},\tilde{y}_0)p(d\hat{x}_1,d\hat{y}_1|\hat{x}_0,u_0,v_0)
p(d\tilde{x}_1,d\tilde{y}_1|\tilde{x}_0,u_0,v_0)\\&\pi^1_1(du_1|\hat{\psi},\hat{y}_0,u_0,v_0,\hat{y}_1)
\pi^2_1(dv_1|\tilde{\psi},\tilde{y}_0,u_0,v_0,\tilde{y}_1)
\cdots\pi^1_{n-1}(du_{n-1}|\hat{\psi},\hat{y}_0,u_0,v_0,\hat{y}_1,\cdots, u_{n-2},v_{n-2},\hat{y}_{n-1})\\&\pi^2_{n-1}(dv_{n-1}|\tilde{\psi},\tilde{y}_0,u_0,v_0,\tilde{y}_1,\cdots, u_{n-2},v_{n-2},\tilde{y}_{n-1})p(d\hat{x}_n,d\hat{y}_n|\hat{x}_{n-1},u_{n-1},v_{n-1})p(d\tilde{x}_n,d\tilde{y}_n|\tilde{x}_{n-1},u_{n-1},v_{n-1})\,.
\end{align*}}
On $(\bar{\Omega},\bar{\mathcal{F}},\bar{\mathbb{P}})$, define the processes $\{\hat{X}_n\}$, $\{\tilde{X}_n\}$, $\{\hat{Y}_n\}$, $\{\tilde{Y}_n\}$, $\{U_n\}$, $\{V_n\}$ canonically. Then the Markov chains $\{\hat{X}_n\}, \{\tilde{X}_n\}$ on $(\bar{\Omega},\bar{\mathcal{F}},\bar{\mathbb{P}})$ form the desired pair. For notational simplicity we omit the superscripts and subscripts on $\bar{\mathbb{P}}$. We denote by $\bar{X_n}=(\hat{X}_n,\tilde{X}_n)$ and the associated observation pair by $\bar{Y}_n=(\hat{Y}_n,\tilde{Y}_n)$. Then $\{\bar{X}_n\}$ is an $X^{2}$ valued Markov chain. Let the controlled transition kernel be denoted by $$\bar{p}(d\bar{z},d\bar{y}|\bar{x},u,v) \in \mathcal{P}(X^2\times Y^2)$$ for $\bar{x}=(x_1,x_2) \in X^2$. Define $G=K^2$ and define $\Theta \in \mathcal{P}(X^2)$ by $$\Theta(A)=\frac{\lambda \times \lambda(A \cap G)}{\lambda(K)^2}$$ for any Borel set $A$ of $X^2$. Then if follows from our assumptions that
$$\bar{p}(A \times Y^2|\bar{x},u,v)\geq \delta I_G(\bar{x})\Theta(A)$$ where $\delta = \frac{1}{2}(\inf_{x \in K, u \in U, v \in V, z \in K}\int_Y\varphi(x,u,v,z,y)\eta(dy)\lambda(K))^2$. This is the minorization condition of \cite{meyn} in the present context which enables us to carry out the Athreya-Ney-Nummelin construction of pseudo-atom \cite{meyn}.

Let $H=X^2$ and $H^*=X^2\times \{0,1\}$. Endow $H^*$ with its Borel $\sigma-$field. For any measure $\mu$ on $H$, define a measure $\mu^*$ on $H^*$ as follows: For Borel $A \subset H$, let $A_0=A \times \{0\}$ and $A_1=A \times \{1\}$. Then
\begin{align*}\mu^*(A_0)&=(1-\delta)\mu(A\cap K^2)+\mu(A\cap(K^2)^c)\\
\mu^*(A_1)&=\delta\mu(A \cap K^2)\,.
\end{align*}For a measure $\mu$ on $H \times Y^2$, we define the measure $\mu^*$ on $H^*\times Y^2$ by
\begin{align*}\mu^*(A_0\times D)&=(1-\delta)\mu((A\cap K^2)\times D)+\mu((A\cap(K^2)^c)\times D)\\
\mu^*(A_1\times D)&=\delta\mu((A \cap K^2)\times D)\,,
\end{align*}for $D \subset Y^2$ Borel. On a suitable probability space $(\Omega^*,\mathcal{F}^*,\mathbb{P}^*)$, define an $H^*$-valued controlled Markov chain $\{X_n^*,i^*_n\}$ (where $X_n^*=(\hat{X}_n^*,\tilde{X}_n^*$)) with $U$- valued control process $\{U^*_n\}$ and $V$- valued control process $\{V_n^*\}$ and $Y^2$-valued observation process $\{Y^*_n\}$, such that:\\
(i) The controlled transition kernel of $\{X_n^*,i_n^*,Y_n^*\}$ is given by: for $x=(x_0,i_0)\in H^*$,
\begin{align*}q(d\bar{x}, d\bar{y}|x,u,v)&=\bar{p}^*(d\bar{x}, d\bar{y}|x_0,u,v),\,\,\,\,\,x\in H_0-K^2\times \{0\}
\\&=\frac{1}{1-\delta}(\bar{p}^*(d\bar{x}, d\bar{y}|x_0,u,v)-\delta\Theta^*(d\bar{x})\eta^2(d\bar{y})),\,\,\,\,x \in K^2\times \{0\}
\\&=\Theta^*(d\bar{x})\eta^2(d\bar{y}),\,\,\,\,\,\,x\in H_1\,,
\end{align*}
(ii) \begin{align*}\mathbb{P}^*((X_0^*,i_0^*)\in A_0,Y^*_0\in A^{\prime}, U_0^* \in \Delta, V_0^* \in \Gamma)= &(1-\delta)\bar{\mathbb{P}}(\bar{X}_0\in A \cap K^2, \bar{Y}_0 \in A^{\prime}, U_0 \in \Delta, V_0 \in \Gamma)\\&+\bar{\mathbb{P}}(\bar{X}_0\in A\cap (K^2)^c, \bar{Y}_0\in A^{\prime}, U_0 \in \Delta, V_0 \in \Gamma)
\end{align*}
\begin{align*}\mathbb{P}^*((X_0^*,i_0^*)\in A_1,Y^*_0\in A^{\prime}, U_0^* \in \Delta, V_0^* \in \Gamma)= \delta\bar{\mathbb{P}}(\bar{X}_0\in A \cap K^2, \bar{Y}_0 \in A^{\prime}, U_0 \in \Delta, V_0 \in \Gamma)
\end{align*}for $A\subset H, A^{\prime} \subset Y^2, \Delta \subset U, \Gamma \subset V$ Borel,\\
(iii) and \begin{align*}&\mathbb{P}^*(U_n^*\in \Delta, V_n^* \in \Gamma|(X_m^*,i_m^*,Y^*_m)=(x_m,i_m,y_m), m\leq n, U_k^*=u_k,V_k^*=v_k,k<n)\\&=\bar{\mathbb{P}}(U_n\in \Delta, V_n \in \Gamma|(\bar{X}_m,\bar{Y}_m)=(x_m,y_m), m\leq n, U_k=u_k,V_k=v_k,k<n)\,\,\mbox{for}\,\,n\geq1\,.
\end{align*}From the above construction the following lemmas can be proved.
\begin{lem}The set $K^2\times \{1\}$ is an accessible atom of $\{(X_n^*,i^*_n)\}$ in the sense of Meyn and Tweedie (\cite{meyn}).
\end{lem}
\begin{lem}For any Borel $A^{i}\subset H,B^{i}\subset Y^2, \Delta^i\subset U, \Gamma^i\subset V,0\leq i\leq n,n\geq 0$
\begin{align*}\mathbb{P}^*&\biggl(((X_0^*,i_0^*,Y_0^*,U_0^*,V_0^*),\cdots,(X_n^*,i_n^*,Y_n^*,U_n^*,V_n^*))\in \prod_{i=0}^n(A^i_0\cup A^i_1)\times B^i\times \Delta^i\times \Gamma^i\biggr)\\&=\bar{\mathbb{P}}\biggl(((\bar{X}_0,\bar{Y}_0,U_0,V_0),\cdots,(\bar{X}_n,\bar{Y}_n,U_n,V_n))\in \prod_{i=0}^n A^i\times B^i\times \Delta^i\times \Gamma^i\biggr)
\end{align*}
\end{lem}
Let \begin{align}\label{stoptime}\tau=\min\{n\geq 0 : (X_n^*,i^*_n)\in K^2\times \{1\}\}\,.\end{align} Then the following lemma can be proved using $(A1)$ and standard arguments as in \cite{meyn}
\begin{lem}Under $(A1)$ we have, \begin{align}\label{bound}\mathbb{E}^*[\tau\,|\,(X_0^*,i_0^*)=(x,i)]=O(\mathcal{V}(x_1)+ \mathcal{V}(x_2))\end{align} for any $(x,i)=((x_1,x_2),i) \in X^2\times \{0,1\}$, where $\tau$ is as in \eqref{stoptime}.
\end{lem}The following lemma gives a bound on the difference of $V_{\alpha}(.)$ for two different values of its argument.
\begin{lem}For $\hat{\psi}, \tilde{\psi} \in \mathcal{P}_0(X)$, there exists a suitable constant $\bar{K}$ such that
$$|V_{\alpha}(\hat{\psi})-V_{\alpha}(\tilde{\psi})|\leq \bar{K}[\int\mathcal{V}d \hat{\psi}+\int\mathcal{V}d \tilde{\psi}]\,.$$
\end{lem}
\begin{proof}Let $V_{\alpha}(\hat{\psi})\geq V_{\alpha}(\tilde{\psi})$. The other case can be handled with a symmetric argument. Let $\pi^1=\{\pi_n^1\}$ be an optimal policy for player 1 for the discounted payoff \textbf{POSG} with initial distribution $\hat{\psi}$ and let $\pi^2=\{\pi_n^2\}$ be an optimal policy for player 2 for the discounted payoff \textbf{POSG} with initial distribution $\tilde{\psi}$. Then we have
\begin{align*}|V_{\alpha}(\hat{\psi})-V_{\alpha}(\tilde{\psi})|&\leq |\sum_{m=0}^{\infty}\alpha^m\bar{\mathbb{E}}^{\pi^1,\pi^2}_{\hat{\psi},\tilde{\psi}}[c(\hat{X}_n,U_n,V_n)]-
\sum_{m=0}^{\infty}\alpha^m\bar{\mathbb{E}}^{\pi^1,\pi^2}_{\hat{\psi},\tilde{\psi}}[c(\tilde{X}_n,U_n,V_n)]|\\
&=|\sum_{m=0}^{\infty}\alpha^m\bar{\mathbb{E}}[c(\hat{X}_n,U_n,V_n)-c(\tilde{X}_n,U_n,V_n)]|\\
&\leq |\mathbb{E}^*\sum_{m=0}^{\tau}\alpha^m[c(\hat{X}_n^*,U^*_n,V^*_n)-c(\tilde{X}_n^*,U^*_n,V^*_n)]|\\
&\leq 2 ||c||_{\infty}\mathbb{E}^*(\tau)
\end{align*}where the third step follows from the fact that $\hat{X}^*_{\tau+m},\tilde{X}_{\tau+m}^*$ for $m \geq 1$ has the same law conditioned on all the information up to time $\tau$. Thus from \eqref{bound} we have
$$|V_{\alpha}(\psi_1)-V_{\alpha}(\psi_2)|\leq \bar{K}\bar{\mathbb{E}}[\mathcal{V}(\hat{X}_0)+\mathcal{V}(\tilde{X}_0)]\,.$$ Hence the lemma follows.
\end{proof}Now fix $\psi^* \in \mathcal{P}_0(X)$. Define $\bar{V}_{\alpha}(\psi)=V_{\alpha}(\psi)-V_{\alpha}(\psi^*)$. Thus substituting in \eqref{spe} we get
\begin{align}\nonumber\bar{V}_{\alpha}(\psi)+ (1-\alpha)V_{\alpha}(\psi^*)&=\min_{\nu \in \mathcal{P}(V)}\max_{\mu\in \mathcal{P}(U)}\biggl[\bar{\tilde{c}}(\psi,\mu,\nu)
+ \alpha \int_{\mathcal{P}_0(X)} \bar{V}_{\alpha}(\psi^{\prime})\phi(d\psi^{\prime}|\psi,\mu,\nu)\biggr]\\
&=\max_{\mu \in \mathcal{P}(U)}\min_{\nu \in \mathcal{P}(V)}\biggl[\tilde{\bar{c}}(\psi,\mu,\nu)
+ \alpha \int_{\mathcal{P}_0(X)}\bar{V}_{\alpha}(\psi^{\prime})\phi(d\psi^{\prime}|\psi,\mu,\nu)\biggr]\,.
\end{align}
Now $(1-\alpha)V_{\alpha}(\psi^*)$ is bounded. Thus we can find an $\alpha(n)\rightarrow 1$ such that \begin{align}\label{constant}(1-\alpha(n))V_{\alpha(n)}(\psi^*)\rightarrow \gamma\end{align} for some $\gamma \in \mathbb{R}$. Let $\hat{V}(\psi)= \displaystyle \limsup_{n \rightarrow \infty}\bar{V}_{\alpha(n)}(\psi)$ and $\underline{V}(\psi)=\displaystyle\liminf_{n \rightarrow \infty}\bar{V}_{\alpha(n)}(\psi)$.
\begin{lem}The function $\hat{V}$ satisfies
\begin{align}\label{ub}\hat{V}(\psi)+ \gamma &\leq \max_{\mu \in \mathcal{P}(U)}\min_{\nu \in \mathcal{P}(V)}\biggl[\bar{\tilde{c}}(\psi,\mu,\nu)+ \int_{\mathcal{P}_0(X)} \hat{V}(\psi^{\prime})\phi(d\psi^{\prime}|\pi,\mu,\nu)\biggr]\,,
\end{align} where $\gamma$ is as in \eqref{constant}.
\end{lem}
\begin{proof}We have
\begin{align*}\bar{V}_{\alpha(n)}(\psi)+ (1-\alpha(n))V_{\alpha(n)}(\psi^*)&= \max_{\mu \in \mathcal{P}(U)}\min_{\nu \in \mathcal{P}(V)}\biggl[\bar{\tilde{c}}(\psi,\mu,\nu)+ \alpha(n) \int_{\mathcal{P}_0(X)} \bar{V}_{\alpha(n)}(\psi^{\prime})\phi(d\psi^{\prime}|\psi,\mu,\nu)\biggr]\,.
\end{align*}Now taking limit $n \rightarrow \infty$ in the above we get
\begin{align*}\hat{V}(\psi)+ \gamma &= \limsup _{n \rightarrow \infty}\max_{\mu \in \mathcal{P}(U)}\min_{\nu \in \mathcal{P}(V)}\biggl[\bar{\tilde{c}}(\psi,\mu,\nu)+ \alpha(n) \int_{\mathcal{P}_0(X)} \bar{V}_{\alpha(n)}(\psi^{\prime})\phi(d\psi^{\prime}|\psi,\mu,\nu)\biggr]\\
&= \limsup _{n \rightarrow \infty}\min_{\nu \in \mathcal{P}(V)}\biggl[\bar{\tilde{c}}(\psi,\mu_n^*,\nu)+ \alpha(n) \int_{\mathcal{P}_0(X)} \bar{V}_{\alpha(n)}(\psi^{\prime})\phi(d\psi^{\prime}|\psi,\mu_n^*,\nu)\biggr]\\
& \leq \min_{\nu \in \mathcal{P}(V)}\limsup _{n \rightarrow \infty}\biggl[\bar{\tilde{c}}(\psi,\mu_n^*,\nu)+ \alpha(n) \int_{\mathcal{P}_0(X)}\bar{V}_{\alpha(n)}(\psi^{\prime})\phi(d\psi^{\prime}|\psi,\mu_n^*,\nu)\biggr]\,.
\end{align*}
In the second step $\mu_n^*$ is the outer maximiser. Now fix $\pi$. By dropping to a subsequence if necessary, we may suppose that $\bar{V}_{\alpha(n)}(\psi)\rightarrow \hat{V}(\psi)$ and $\mu_n^*\rightarrow \mu^*$ in $\mathcal{P}(U)$. Now by previous lemma $|\bar{V}_{\alpha}(\psi)|\leq K_1(1+\int \mathcal{V}d\psi)$. Thus by Lemma 8.3.7 in \cite{lerma}, the last expression is bounded above by
\begin{align*}&\min_{\nu \in \mathcal{P}(V)}\biggl[\bar{\tilde{c}}(\psi,\mu^*,\nu)+ \int_{\mathcal{P}_0(X)} \hat{V}(\psi^{\prime})\phi(d\psi^{\prime}|\psi,\mu_n^*,\nu)\biggr]\\
&\leq \max_{\mu \in \mathcal{P}(U)}\min_{\nu \in \mathcal{P}(V)}\biggl[\bar{\tilde{c}}(\psi,\mu,\nu)+ \int_{\mathcal{P}_0(X)} \hat{V}(\psi^{\prime})\phi(d\psi^{\prime}|\psi,\mu,\nu)\biggr]\,.
\end{align*} The claim follows.
\end{proof}
Similarly we have the following result.
\begin{lem}The function $\underline{V}$ satisfies
\begin{align}\label{lb}\underline{V}(\psi)+ \gamma &\geq \min_{\nu \in \mathcal{P}(V)}\max_{\mu \in \mathcal{P}(U)}\biggl[\bar{\tilde{c}}(\psi,\mu,\nu)+ \int_{\mathcal{P}_0(X)}\underline{V}(\psi^{\prime})\phi(d\psi^{\prime}|\psi,\mu,\nu)\biggr]\,.
\end{align}
\end{lem}
\begin{proof}We have
\begin{align*}\bar{V}_{\alpha(n)}(\psi)+ (1-\alpha(n))V_{\alpha(n)}(\psi^*)&= \min_{\nu \in \mathcal{P}(V)}\max_{\mu \in \mathcal{P}(U)}\biggl[\bar{\tilde{c}}(\psi,\mu,\nu)+ \alpha(n) \int_{\mathcal{P}_0(X)} \bar{V}_{\alpha(n)}(\psi^{\prime})\phi(d\psi^{\prime}|\psi,\mu,\nu)\biggr]\,.
\end{align*}Now taking limit $n \rightarrow \infty$ in the above we get
\begin{align*}\underline{V}(\psi)+ \gamma &= \liminf _{n \rightarrow \infty}\min_{\nu \in \mathcal{P}(V)}\max_{\mu \in \mathcal{P}(U)}\biggl[\bar{\tilde{c}}(\psi,\mu,\nu)+ \alpha(n) \int_{\mathcal{P}_0(X)} \bar{V}_{\alpha(n)}(\psi^{\prime})\phi(d\psi^{\prime}|\psi,\mu,\nu)\biggr]\\
&= \liminf _{n \rightarrow \infty}\max_{\mu \in \mathcal{P}(U)}\biggl[\bar{\tilde{c}}(\psi,\mu,\nu^*_n)+ \alpha(n) \int_{\mathcal{P}_0(X)}\bar{V}_{\alpha(n)}(\psi^{\prime})\phi(d\psi^{\prime}|\psi,\mu,\nu^*_n)\biggr]\\
& \geq \max_{\mu \in \mathcal{P}(U)}\liminf _{n \rightarrow \infty}\biggl[\bar{\tilde{c}}(\psi,\mu,\nu^*_n)+ \alpha(n) \int_{\mathcal{P}_0(X)}\bar{V}_{\alpha(n)}(\psi^{\prime})\phi(d\psi^{\prime}|\psi,\mu,\nu^*_n)\biggr]\,.
\end{align*}
In the second step $\nu_n^*$ is the outer minimiser. Now by arguments analogous to the proof of the above lemma we have that there exists a $\nu^* \in \mathcal{P}(V)$ such that the last expression is bounded below by
\begin{align*}&\max_{\mu \in \mathcal{P}(U)}\biggl[\bar{\tilde{c}}(\psi,\mu,\nu^*)+ \int_{\mathcal{P}_0(X)} \underline{V}(\psi^{\prime})\phi(d\psi^{\prime}|\psi,\mu,\nu^*)\biggr]\\
&\geq \min_{\nu \in \mathcal{P}(V)}\max_{\mu \in \mathcal{P}(U)}\biggl[\bar{\tilde{c}}(\psi,\mu,\nu)
+ \int_{\mathcal{P}_0(X)}\underline{V}(\psi^{\prime})\phi(d\psi^{\prime}|\psi,\mu,\nu)\biggr]\,.
\end{align*} The claim follows.
\end{proof}
Finally we get the following theorem:
\begin{thm}\label{thm1}Assume \textbf{(A1-A2)}. Then $\gamma$ (as in \eqref{constant}) is the value of the \textbf{COSG}. Moreover let $u^*:\mathcal{P}_0(X)\rightarrow \mathcal{P}(U)$ be a measurable function such that $u^*(.)$ is the outer maximiser of the righthand side of \eqref{ub} (exists by our assumptions and a standard measurable selection theorem). Then the strategy $\{{\pi^*_n}^1\}$ given by $${\pi^*_n}^1(\cdot|i_n)= u^*(\psi_n)(\cdot)$$ is an optimal strategy for player 1. Similarly let $v^*:\mathcal{P}_0(X)\rightarrow \mathcal{P}(V)$ be a measurable function such that $v^*(.)$ is the outer minimiser of the righthand side of \eqref{lb}. Then $\{{\pi^*_n}^2\}$ given by $${\pi^*_n}^2(\cdot|i_n)= v^*(\psi_n)(\cdot)$$ is an optimal strategy for player 2.
\end{thm}
\begin{proof}Let $\{\pi^2_n\}$ be an arbitrary admissible strategy of player 2.
Then we have from \eqref{ub}
\begin{align*}\mathbb{E}^{{\pi^*}^1,\pi^2}_{\psi}[\hat{V}(\Psi_n)]+ \gamma &\leq \mathbb{E}^{{\pi^*}^1,\pi^2}_{\psi}[\tilde{c}(\Psi_n,U_n,V_n)]+\mathbb{E}^{{\pi^*}^1,\pi^2}_{\psi}[\hat{V}(\Psi_{n+1})],\,\,n\geq 0
\end{align*}Therefore we have
\begin{align*}\gamma \leq \frac{1}{n}\sum_{m=0}^{n-1}\mathbb{E}^{{\pi^*}^1,\pi^2}_{\psi}[\tilde{c}(\Psi_m,U_m,V_m)]+\frac{\mathbb{E}^{{\pi^*}^1,\pi^2}_{\psi}
[\hat{V}(\Psi_n)]-\hat{V}(\psi)}{n}\,.
\end{align*}Then by taking limit $n\rightarrow \infty$ we have using assumption (A2)
$$\gamma \leq \liminf_{n\rightarrow \infty}\frac{1}{n}\sum_{m=0}^{n-1}\mathbb{E}^{{\pi^*}^1,\pi^2}_{\psi}[\tilde{c}(\Psi_m,U_m,V_m)]\,.$$
Similarly, if $\{\pi^1_n\}$ is an arbitrary admissible strategy for player 1, then we have by \eqref{lb}
\begin{align*}\mathbb{E}^{\pi^1,{\pi^*}^2}_{\psi}[\underline{V}(\Psi_n)]+ \gamma &\geq \mathbb{E}^{\pi^1,{\pi^*}^2}_{\psi}[\tilde{c}(\Psi_n,U_n,V_n)]+\mathbb{E}^{\pi^1,{\pi^*}^2}_{\psi}[\underline{V}(\Psi_{n+1})],\,\,n\geq 0\,.
\end{align*}Therefore we have
\begin{align*}\gamma \geq \frac{1}{n}\sum_{m=0}^{n-1}\mathbb{E}^{\pi^1,{\pi^*}^2}_{\psi}[\tilde{c}(\Psi_m,U_m,V_m)]+\frac{\mathbb{E}^{\pi^1,{\pi^*}^2}_{\psi}
[\underline{V}(\Psi_n)]- \underline{V}(\psi)}{n}\,.
\end{align*}Then by taking limit $n\rightarrow \infty$ we have using assumption (A2)
$$\gamma \geq \liminf_{n\rightarrow \infty}\frac{1}{n}\sum_{m=0}^{n-1}\mathbb{E}^{\pi^1,{\pi^*}^2}_{\psi}[\tilde{c}(\Psi_m,U_m,V_m)]\,.$$
Now the conclusions follow.
\end{proof}Now the following theorem follows from Theorem \ref{thm1} and the equivalence of \textbf{COSG} and \textbf{POSG}.
\begin{thm}The \textbf{POSG} with average cost criterion has a value and is equal to $\gamma$ (as in \eqref{constant}) for any initial distribution. Moreover, $(\{{\pi_n^*}^1\},\{{\pi_n^*}^2\})$ given by Theorem $\ref{thm1}$ is a saddle point equilibrium.
\end{thm}
\section{Conclusion} In this article we study a partially observed stochastic game under average payoff criterion. We estimate the unobservable state variable and use the state estimate as our new observable state variable . We then use the vanishing discount approach to solve the average cost problem. Our analysis involves a coupling argument which uses the machinery of pseudo-atom construction. We show that the game has a value and also prove the existence of a saddle point equilibrium for our partially observable model.\\

\vspace{.8mm} \noindent \textbf{Acknowledgement.}The author wish
to thank V. S. Borkar and M. K. Ghosh for many helpful discussions
and comments.

\end{document}